\documentclass[a4paper,11pt]{article}
\usepackage[top=3cm,bottom=3cm,left=2.5cm,right=2.5cm]{geometry}
\usepackage[T1]{fontenc}
\usepackage[utf8]{inputenc}
\usepackage{lmodern,microtype}
\usepackage[english]{babel}
\usepackage{amsmath,amssymb,amsthm,mathtools}
\usepackage{setspace,enumitem,booktabs,array,longtable,calc,aliascnt,cite,needspace}
\usepackage[hidelinks,hypertexnames=false]{hyperref}
\usepackage[nameinlink,capitalize,noabbrev]{cleveref}
\setlist[enumerate]{leftmargin=2.2em,itemsep=0.12em,topsep=0.25em}
\setlist[itemize]{leftmargin=2em,itemsep=0.12em,topsep=0.25em}
\allowdisplaybreaks
\let\originalthebibliography\thebibliography
\renewcommand{\thebibliography}[1]{\originalthebibliography{#1}\setlength{\itemsep}{2pt}\interlinepenalty=10000}
\newtheorem{lemma}{Lemma}[section]
\newaliascnt{theorem}{lemma}
\newtheorem{theorem}[theorem]{Theorem}
\aliascntresetthe{theorem}
\newaliascnt{proposition}{lemma}
\newtheorem{proposition}[proposition]{Proposition}
\aliascntresetthe{proposition}
\newaliascnt{corollary}{lemma}
\newtheorem{corollary}[corollary]{Corollary}
\aliascntresetthe{corollary}
\newcommand{\cF}{\mathcal F}
\newcommand{\cK}{\mathcal K}
\newcommand{\cQ}{\mathcal Q}
\DeclareMathOperator{\dist}{dist}
\newcommand{\tightlist}{\setlength{\itemsep}{0pt}\setlength{\parskip}{0pt}}
\newcommand{\proofpart}[1]{\par\smallskip\Needspace{3\baselineskip}\noindent\emph{#1.}\quad}

\hypersetup{pdftitle={Degree-list Kempe equivalence for 3-connected graphs},pdfauthor={Hongyang Wang},pdfsubject={Graph coloring and Kempe equivalence}}
\title{Degree-list Kempe equivalence for 3-connected graphs}
\author{Hongyang Wang\thanks{Email address: \href{mailto:hwanghj@connect.ust.hk}{\text{hwanghj@connect.ust.hk}}}\\[0.5em]
\normalsize Department of Mathematics,\\
\normalsize Hong Kong University of Science and Technology,\\
\normalsize Clear Water Bay, Hong Kong}
\date{October 26, 2026}
\begin{document}
\maketitle
\begin{abstract}
In a proper coloring, a Kempe change interchanges two colors on a connected component of the subgraph they induce. For a fixed list assignment $L$, a Kempe change is $L$-valid if it transforms an $L$-coloring into another $L$-coloring. We prove that if $G$ is a $3$-connected graph other than a complete graph or the triangular prism, and $L$ is a list assignment with $|L(v)|\ge d_G(v)$ for every vertex $v$, then $G$ has a proper $L$-coloring and any two such colorings can be transformed into one another by a finite sequence of $L$-valid Kempe changes. This proves a conjecture of Chakraborty, Feghali and Mahmoud~\cite{CFM}.
\end{abstract}
\noindent\textbf{Mathematics Subject Classification (2020)}: 05C15.

\noindent\textbf{Keywords}: List coloring; Kempe equivalence; vertex connectivity; Gallai trees.
\section{Introduction}\label{sec:introduction}

All graphs are finite, simple and undirected. For a graph \(G\) and a set \(X\subseteq V(G)\), let \(G[X]\) be the induced subgraph on \(X\) and let \(G-X=G[V(G)\setminus X]\). We identify an induced subgraph with its vertex set when no ambiguity arises. We write \(N_G(v)\) and \(d_G(v)\) for the neighborhood and degree of \(v\), and \(\Delta(G)\) and \(\delta(G)\) for the maximum and minimum degrees. Also, \(N_G(X)=\bigcup_{v\in X}N_G(v)\). Subscripts are omitted when the graph is clear. Two disjoint vertex sets are \emph{anticomplete} if no edge joins them.

A \emph{list assignment} \(L\) assigns a finite set of colors \(L(v)\) to each vertex \(v\). A coloring is \emph{proper} if adjacent vertices have different colors. An \emph{\(L\)-coloring} is a proper coloring \(\varphi\) satisfying \(\varphi(v)\in L(v)\) for every vertex, and \(G\) is \emph{\(L\)-colorable} if it has one. An ordinary \emph{\(k\)-coloring} is the special case \(L(v)=\{1,\ldots,k\}\) for every vertex; the chromatic number \(\chi(G)\) is the least \(k\) for which such a coloring exists. A \emph{\(k\)-assignment} has lists of size \(k\) at every vertex. A \emph{degree-assignment} satisfies \(|L(v)|\ge d_G(v)\) at every vertex, and is \emph{tight} if equality holds everywhere.

For a proper coloring \(\varphi\) and distinct colors \(a,b\), let \(G_\varphi(a,b)\) be the subgraph induced by the vertices colored \(a\) or \(b\). Its connected components are the \emph{\((a,b)\)-components} of \(\varphi\), also called \emph{Kempe chains}. A \emph{Kempe change} interchanges \(a\) and \(b\) on one such component. Either color may be absent from the current coloring.

A Kempe change is \emph{\(L\)-valid} if it transforms an \(L\)-coloring into another \(L\)-coloring. Two \(L\)-colorings are \emph{\(L\)-equivalent}, written \(\varphi\sim_L\psi\), if one can be transformed into the other by a finite sequence of \(L\)-valid Kempe changes. The lists remain fixed throughout the sequence. For ordinary \(k\)-colorings, we say simply \emph{Kempe equivalent}. We call the equivalence classes \emph{Kempe classes}. Following Cranston and Mahmoud~\cite{CM}, we call \(G\) \emph{\(L\)-swappable} if it has an \(L\)-coloring and all its \(L\)-colorings are \(L\)-equivalent, and \emph{degree-swappable} if this holds for every degree-assignment.

Kempe introduced these changes in his work on the four-color problem~\cite{Kempe}. For the broader context of reconfiguration, see Nishimura~\cite{Nishimura}.

We first recall results for ordinary colorings. A graph is \emph{\(d\)-degenerate} if every subgraph with at least one vertex has a vertex of degree at most \(d\). Las Vergnas and Meyniel~\cite{LVM} proved that all \(k\)-colorings of a \(d\)-degenerate graph are Kempe equivalent whenever \(k>d\). Mohar~\cite{Mohar} asked whether all \(k\)-colorings of a connected noncomplete \(k\)-regular graph are Kempe equivalent. Feghali, Johnson and Paulusma~\cite{FJP} settled the \(3\)-regular case, with the triangular prism as the only further exception. The triangular prism is the Cartesian product \(K_2\square K_3\): it consists of two disjoint triangles \(a_1a_2a_3\) and \(b_1b_2b_3\), together with the edges \(a_ib_i\) for \(i=1,2,3\). Bonamy, Bousquet, Feghali and Johnson~\cite{BBFJ} proved the assertion for every \(k\ge4\).

For planar graphs, Mohar~\cite{Mohar} proved that all \(k\)-colorings of a planar graph \(G\) are Kempe equivalent when \(k>\chi(G)\). Feghali~\cite{FeghaliPlanar} proved that all \(4\)-colorings of a \(4\)-critical planar graph are Kempe equivalent; here \(4\)-critical means that the graph has chromatic number \(4\) and every proper subgraph is \(3\)-colorable. Deschamps, Feghali, Kardo\v{s}, Legrand-Duchesne and Pierron~\cite{Deschamps} proved that any \(5\)-coloring of a planar graph can be transformed into any other by a number of Kempe changes bounded by a polynomial in the number of vertices.

Quantitative and algorithmic questions have also been studied. Bonamy, Delecroix and Legrand-Duchesne~\cite{BDL} proved that, for fixed \(\Delta\ge3\), any \(\Delta\)-coloring of a connected \(n\)-vertex graph of maximum degree \(\Delta\), other than a complete graph or the triangular prism, can be transformed into any other by at most \(C_\Delta n^2\) Kempe changes, where \(C_\Delta\) depends only on \(\Delta\). On general graphs, deciding whether two given \(k\)-colorings are Kempe equivalent is PSPACE-complete for every fixed \(k\ge3\), as shown by Bonamy, Heinrich, Ito, Kobayashi, Mizuta, M\"{u}hlenthaler, Suzuki and Wasa~\cite{Diameter}.\footnote{PSPACE consists of decision problems solvable with polynomial workspace. A problem is PSPACE-complete if it belongs to PSPACE and every problem in PSPACE has a polynomial-time reduction to it.}

We now turn to list colorings. The classical degree-list colorability results characterize the connected graphs colorable from every degree-assignment; see Erd\H{o}s, Rubin and Taylor~\cite{ERT}. We recall the obstruction characterization in \cref{sec:preliminaries}. Degree-swappability requires, in addition, that any two colorings be transformable into one another by valid Kempe changes.

The line graph of \(G\) has vertex set \(E(G)\), with two vertices adjacent exactly when the corresponding edges share an endpoint. Cranston~\cite{CranstonEdge} proved that, for a planar graph \(G\) of maximum degree \(\Delta\), any two \(L\)-colorings of its line graph are \(L\)-equivalent for every \((\Delta+1)\)-assignment \(L\) when \(\Delta\ge9\), and for every \(\Delta\)-assignment \(L\) when \(\Delta\ge15\).

Cranston and Mahmoud~\cite{CM} proved that a connected graph \(G\) of maximum degree \(\Delta\ge3\) is \(L\)-swappable for every \(\Delta\)-assignment \(L\), unless \(G\) is \(K_{\Delta+1}\) or the triangular prism. They also proved swappability for the two exceptional graphs whenever the lists are not identical, and asked which graphs are degree-swappable. Their results include \(2\)-connected graphs that are not degree-swappable. The distinction between a degree-assignment and a \(\Delta\)-assignment matters for irregular graphs: the former permits smaller lists at vertices of smaller degree.

Chakraborty, Feghali and Mahmoud~\cite{CFM} proved that every \(4\)-connected noncomplete graph is degree-swappable. They conjectured that every \(3\)-connected graph other than a complete graph or the triangular prism is degree-swappable. We prove this conjecture.

\begin{theorem}\label{thm:main}
Let $G$ be a $3$-connected graph that is neither complete nor the triangular prism $K_2\square K_3$. For every finite list assignment $L$ satisfying $|L(v)|\ge d_G(v)$ for all $v\in V(G)$, the graph $G$ has an $L$-coloring, and all its $L$-colorings are $L$-equivalent.
\end{theorem}

Both exceptions are necessary. Identical lists of size \(n-1\) make \(K_n\) uncolorable. With identical lists \(\{1,2,3\}\), every bichromatic subgraph of the triangular prism is connected, so each Kempe change permutes the color names globally. A fixed coloring of one triangle extends in two ways to the prism, and these two colorings are not related by a global color permutation. Hence they are not \(L\)-equivalent.

\Needspace{5\baselineskip}
Our proof builds on the special-color lemmas of Chakraborty, Feghali and Mahmoud~\cite{CFM} and the induced-subgraph extension results of Cranston and Mahmoud~\cite{CM}. The additional difficulty at connectivity three is that deleting three vertices may disconnect the graph. We handle the components that are uncolorable from the reduced lists through their block structure and use recolorings of terminal blocks to remove the obstructions to a valid Kempe change.

The list-coloring tools and notation are collected in \cref{sec:preliminaries}. We prove a reduction at a vertex of maximum degree in \cref{sec:reduction}, and a theorem for two nonadjacent vertices with a common special color in \cref{sec:common}. These results are combined in \cref{sec:main-proof} to prove \cref{thm:main}.

\section{List-coloring tools and notation}\label{sec:preliminaries}

A \emph{block} is a maximal \(2\)-connected subgraph, a bridge together with its ends, or an isolated vertex. The \emph{block-cut tree} of a connected graph has a vertex for each block and each cutvertex, with adjacency given by containment. In a graph with more than one block, an \emph{endblock} contains exactly one cutvertex; its \emph{interior} consists of the other vertices. A \emph{Gallai tree} is a connected graph whose blocks are cliques or odd cycles.

When the blocks are rooted at a vertex \(r\), the \emph{entry} \(e_B\) of a block \(B\) is \(r\) if \(r\in B\), and otherwise is the cutvertex of \(B\) toward \(r\) in the block-cut tree. The block containing a nonroot cutvertex on its side toward \(r\) is its \emph{parent block}; the other incident blocks are its \emph{child blocks}.

A list assignment \(M\) on a Gallai tree \(H\) is \emph{blockwise uniform} if there are color sets \(S_B\), one for each block \(B\), such that
\[
 |S_B|=\begin{cases}|V(B)|-1,&B\text{ is a clique},\\2,&B\text{ is an odd cycle},\end{cases}
 \qquad M(u)=\bigcup_{B\ni u}S_B,
\]
and \(S_B\cap S_{B'}=\varnothing\) whenever distinct blocks \(B,B'\) share a cutvertex. We call \(S_B\) the palette of \(B\). The triangle may be treated as a clique. A path has \emph{order} equal to its number of vertices and \emph{length} equal to its number of edges. In particular, deleting a vertex of an odd cycle leaves a path of even order and odd length.

The following characterization of uncolorable degree-assignments is stated in the form used by Chakraborty, Feghali and Mahmoud~\cite{CFM}, who attribute it to Borodin, Erd\H{o}s, Rubin and Taylor, and Gallai.

\begin{lemma}[{\cite{CFM}}]\label{lem:gallai}
Let $H$ be connected and let $M$ be a degree-assignment on $H$. Then $H$ has no $M$-coloring if and only if $H$ is a Gallai tree and $M$ is blockwise uniform.
\end{lemma}

We say that \(L\) has \emph{strict surplus at \(x\)} if \(|L(x)|>d_G(x)\). The following result handles this case.

\begin{lemma}[{\cite{CFM,CM}}]\label{lem:surplus}
Let $G$ be connected and let $L$ be a degree-assignment. If $|L(x)|>d_G(x)$ for some vertex $x$, then $G$ is $L$-swappable.
\end{lemma}

For existence, color descendants before parents in a spanning tree rooted at \(x\): each nonroot vertex has an uncolored parent, and \(x\) has strict surplus.

Cranston and Mahmoud~\cite{CM} use degree-assignments with lists of size exactly \(d_G(v)\). We apply their results to tight assignments; \cref{lem:surplus} handles strict surplus.

\begin{theorem}[{\cite{CM}}]\label{thm:regular}
Let $G$ be connected, let $\Delta=\Delta(G)\ge3$, and let $L$ be a $\Delta$-assignment. If $G\notin\{K_{\Delta+1},K_2\square K_3\}$, then $G$ is $L$-swappable. If $G$ is one of these two graphs, the same conclusion holds whenever the lists are not identical.
\end{theorem}

A \emph{theta graph} is the union of three internally vertex-disjoint paths with the same two distinct ends, with no other edges. Let \(K_4^+\) denote a graph obtained from \(K_4\) by subdividing one edge at least once.

\begin{lemma}[{\cite{CM}}]\label{lem:kernels}
Every bipartite theta graph and every $K_4^+$ is degree-swappable for tight degree-assignments. If a connected graph $G$ contains either as an induced subgraph, then $G$ is degree-swappable.
\end{lemma}

The induced-subgraph extension result of Cranston and Mahmoud~\cite{CM} transfers degree-swappability from an induced subgraph with at least one vertex to a connected graph for tight assignments. The remaining assignments are covered by \cref{lem:surplus}.

A color \(a\) is \emph{special for \(x\)} if \(a\in L(x)\) and \(a\notin L(u)\) for some neighbor \(u\) of \(x\); we also say that \(x\) is \emph{\(a\)-special}. A vertex is \emph{special} if it has a special color, and \emph{nonspecial} otherwise. Thus a nonspecial vertex satisfies \(L(x)\subseteq L(u)\) for every neighbor \(u\). Specialness depends only on the lists, not on the coloring.

For \(a\in L(x)\), let \(\cF(x,a)\) be the set of \(L\)-colorings that assign \(a\) to \(x\). When we say that all colorings in this set are equivalent, the connecting sequences are allowed to leave the set.

\begin{lemma}[{\cite{CFM}}]\label{lem:fiber}
Let $G$ be $2$-connected, let $L$ be a degree-assignment, and let $a$ be special for $x$. Then $\cF(x,a)$ is nonempty, and all its members are $L$-equivalent.
\end{lemma}

\begin{proof}
On $H=G-x$, delete $a$ from the list of each neighbor of $x$, leaving other lists unchanged. Denote the resulting assignment by $M$. It is a degree-assignment, and a neighbor $z$ with $a\notin L(z)$ satisfies $|M(z)|>d_H(z)$. Since $H$ is connected, \cref{lem:surplus} gives an $M$-coloring and makes all $M$-colorings equivalent. Each extends by $x=a$.

These changes lift to $G$ with $x=a$ fixed. If a changed component uses two colors different from $a$, it cannot join $x$. If it uses $a$ and another color, it contains no neighbor of $x$: such a neighbor cannot already have color $a$, and changing its other color to $a$ would violate its reduced list. Thus the component is also a component of the same bichromatic subgraph of $G$. Every coloring in $\cF(x,a)$ restricts to an $M$-coloring, which proves the lemma.
\end{proof}

\begin{lemma}[{\cite{CFM}}]\label{lem:pair}
Let $G$ be $3$-connected and let $L$ be a degree-assignment. If $x,y$ are nonadjacent special vertices, with special colors $a,b$, respectively, then $\cF(x,a)\cap\cF(y,b)$ is nonempty and all colorings in $\cF(x,a)\cup\cF(y,b)$ are $L$-equivalent. The colors $a$ and $b$ may be equal.
\end{lemma}

\begin{proof}
Delete $x,y$ and remove $a$ from each remaining neighbor list of $x$ and $b$ from each remaining neighbor list of $y$. The resulting assignment is a degree-assignment on the connected graph $G-\{x,y\}$. A neighbor $z$ of $x$ omitting $a$ is not $y$, and loses fewer colors than neighbors. Thus its reduced list has strict surplus. By \cref{lem:surplus}, a reduced-list coloring exists and extends with $x=a,y=b$. The equivalence assertion follows by applying \cref{lem:fiber} at $x$ and at $y$.
\end{proof}

Fix an \(L\)-coloring \(\varphi\), a vertex \(v\), its color \(p=\varphi(v)\) and a distinct color \(a\). Let \(K_\varphi(v;p,a)\) be the component containing \(v\) in \(G_\varphi(p,a)\). For \(i=1,2\), define
\[
 B_i(\varphi;v,p,a)=\{u\in V(K_\varphi(v;p,a)):
 \dist_{K_\varphi(v;p,a)}(v,u)=i,\ 
 \sigma_{p,a}(\varphi(u))\notin L(u)\},
\]
where \(\sigma_{p,a}\) interchanges \(p\) and \(a\). Thus \(B_i\) records the vertices at component distance \(i\) whose lists prohibit the proposed change. In particular, \(B_1=\{u\in N(v):\varphi(u)=a,\ p\notin L(u)\}\), and every vertex of \(B_2\) is colored \(p\) and omits \(a\) from its list. We suppress the parameters when they are fixed. All distances in these sets are measured in the indicated bichromatic component, not in the ambient graph. When the coloring changes, so do the component and the sets.

Allowing a Kempe change on a collection of bichromatic components, as in~\cite{CFM}, gives the same equivalence relation. Swapping one component leaves the others unchanged, so any valid simultaneous change can be performed one component at a time.

\Needspace{18\baselineskip}
\section{A reduction at a vertex of maximum degree}\label{sec:reduction}

\begin{proposition}\label{prop:reduction}
Let $G$ be a finite simple 3-connected graph, let $L$ be a tight degree-assignment, let $v$ be a maximum-degree vertex, and let $a\in L(v)$. Let $\varphi$ be a proper $L$-coloring, put $p=\varphi(v)\ne a$, and measure obstruction sets in the whole current component $K_\varphi(v;p,a)$. If $B_1(\varphi;v,p,a)=\varnothing$ and $B_2(\varphi;v,p,a)\ne\varnothing$, then at least one of the following holds.

\begin{enumerate}[label=\textnormal{(\roman*)}]
\item $\varphi$ is $L$-equivalent to a coloring $\psi$ with $\psi(v)=p$ and $B_1(\psi;v,p,a)=B_2(\psi;v,p,a)=\varnothing$.
\item $G$ has two nonadjacent special vertices for which $a$ is a common special color.
\item $\varphi$ is $L$-equivalent to a coloring in $\cF(v,a)$.
\item $G$ contains an induced bipartite theta graph or an induced $K_4^+$; hence $G$ is degree-swappable.
\end{enumerate}
\end{proposition}

The color \(a\) need not be special for \(v\). The proof occupies the remainder of this section.

\subsection{Reduced lists and the one-block cases}\label{sec:reduced-lists}

Choose a shortest two-edge path \(v-x-z\) in the current whole \((p,a)\)-component with \(\varphi(x)=a\), \(\varphi(z)=p\), and \(a\notin L(z)\).

Such a path exists because \(B_2\ne\varnothing\). Properness gives \(vz\notin E(G)\), and \(B_1=\varnothing\) gives \(p\in L(x)\). Hence \(x\) is \(a\)-special, witnessed by \(z\).

Put \(H=G-\{v,x\}\). It is connected. For \(w\in V(H)\) define \(M(w)\) by the following table; boundary types refer to adjacency in \(G\).

\begin{center}
\begin{tabular}{@{}p{0.27\linewidth}p{0.64\linewidth}@{}}
\toprule
Adjacency of $w$ & Colors deleted from $L(w)$ \\
\midrule
To $v$ only & $p$ if $p\in L(w)$; otherwise $a$ \\
To $x$ only & $a$ if $a\in L(w)$; otherwise $p$ \\
To both $v,x$ & Both $p$ and $a$ \\
To neither & None \\
\bottomrule
\end{tabular}
\end{center}

At most one color is deleted for each lost boundary neighbor, so \(M\) is a degree-assignment on \(H\). The table also gives:

\begin{itemize}
\tightlist
\item
  no \(v\)-neighbor retains \(p\) in its \(M\)-list;
\item
  no \(x\)-neighbor retains \(a\) in its \(M\)-list;
\item
  if an \(M\)-colored \(v\)-neighbor receives \(a\), then it is adjacent to \(v\) only and \(p\) belongs to its original list.
\end{itemize}

Suppose first that \(H\) is \(M\)-colorable. Extend an \(M\)-coloring by \(v=p\) and \(x=a\). It is proper. Both it and the original coloring are in \(\cF(x,a)\), so \cref{lem:fiber} makes them equivalent. The first obstruction set is empty by the third property above.

If a vertex of \(B_2\) remains on a path \(v-y-w\), then \(y\ne x\). Indeed, if \(w\) were adjacent to \(x\) and omitted \(a\), the \(x\)-only rule would delete \(p\), while the common-neighbor rule also deletes \(p\); hence \(w\) could not be colored \(p\). Thus \(x\) and \(y\) are distinct \(a\)-colored vertices and therefore nonadjacent. The vertex \(x\) is \(a\)-special through \(z\), and \(y\) is \(a\)-special through \(w\). This gives outcome (i) or outcome (ii).

We may therefore assume that \(H\) is not \(M\)-colorable. By \cref{lem:gallai}, \(H\) is a Gallai tree and \(M\) is blockwise uniform. In particular, \(|M(w)|=d_H(w)\) for every \(w\in V(H)\).

Every deletion in the table is consequently exact. At a one-boundary vertex exactly one listed color is removed; at a common boundary vertex both \(p\) and \(a\) are present and removed; and at every other vertex nothing is removed. At \(z\), we have \(M(z)=L(z)\setminus\{p\}\) and \(a\notin L(z)\).

We first consider the case in which \(H\) has only one block. Since \(H\) is a Gallai tree, it is then a clique or an odd cycle.

\proofpart{\(H\) is a clique}

Suppose \(H=K_m\). The witness \(z\) has degree \(m\) in \(G\): it sees the other \(m-1\) vertices of \(H\) and \(x\), but not \(v\). Since \(v\) has maximum degree, \(d_G(v)\ge m\). Since \(v\) misses \(z\), it has at most \(m\) neighbors, so equality holds and \(v\) sees every vertex of \(H\) except \(z\).

If \(x\) had any other neighbor \(w\) in \(H\), then \(w\) would see all other vertices of \(H\), together with \(v\) and \(x\), and would have degree \(m+1>d_G(v)\). Hence \(N_G(x)=\{v,z\}\), contradicting 3-connectivity. Thus the one-block clique case cannot occur.

\proofpart{\(H\) is an odd cycle}

Assume \(H=C_n\) with odd \(n\ge5\); the triangle is already included in the complete case. Let its common two-color block palette be \(S=\{b_0,b_1\}.\)

The witness \(z\) satisfies \(L(z)=S\mathbin{\dot\cup}\{p\}\), so \(p,a\notin S\). Every rim vertex must see \(x\) or \(v\), because its rim degree is two and \(\delta(G)\ge3\). Call a rim vertex type \(X\), type \(V\), or type \(B\) according as it sees only \(x\), only \(v\), or both. The witness \(z\) has type \(X\). Exactness gives an original list \(S\) plus one of \(p,a\) at an exclusive vertex, and \(S\cup\{p,a\}\) at a type-\(B\) vertex.

Color \(z\) with \(p\) and color the even-order path \(C_n-z\) alternately from \(S\). Together with \(v=p\) and \(x=a\), this gives a proper coloring \(\chi\). Since \(x\) remains \(a\)-colored, \cref{lem:fiber} gives \(\varphi\sim_L\chi\). There is no \(a\)-colored rim vertex and \(z\) is the unique \(p\)-colored rim vertex.

If \(h\in L(x)\setminus(S\cup\{p,a\}),\) then \(h\) is absent from all neighbors of \(x\). The \((a,h)\)-component of \(x\) is the singleton \(\{x\}\), so change \(x\) to \(h\). No vertex is then colored \(a\), and the \((p,a)\)-component of \(v\) is the singleton \(\{v\}\); change \(v\) to \(a\). This is outcome (iii).

Assume no such \(h\) exists. Tightness and minimum degree imply that \(x\) has two or three rim neighbors. If it has two, then \(L(x)=\{p,a,h\}\) for one \(h\in S\). Choose the alternating orientation so that the rim neighbor of \(x\) other than \(z\) avoids \(h\). The same two singleton changes give outcome (iii).

Thus \(x\) has three rim neighbors \(z,u,w\) and \(L(x)=S\cup\{p,a\}\). If \(u,w\) receive the same color, the other color of \(S\) is free at \(x\), and the preceding argument applies. Hence \(u,w\) receive opposite colors.

For \(s\in\{u,w\}\), put \(h=\chi(s)\). The \((a,h)\)-component at \(x\) is exactly the edge \(xs\): the vertex \(x\) is the only \(a\)-colored vertex and \(s\) is its only \(h\)-colored neighbor.

If \(s\) has type \(X\) and lists \(a\), swap on \(xs\). The only \(a\)-colored vertex is then a nonneighbor of \(v\), so \(v\) is a singleton \((p,a)\)-component and can be changed to \(a\).

If \(s\) has type \(B\) and is not adjacent to \(z\), swap on \(xs\). Exactness gives \(p,a\in L(s)\), and the \((p,a)\)-component at \(v\) is exactly the edge \(vs\), because \(z\) is the only other \(p\)-colored rim vertex. A second swap sends \(v\) to \(a\).

Assume neither \(u\) nor \(w\) permits these changes. Every type-\(B\) vertex among them is adjacent to \(z\), and every type-\(X\) vertex among them omits \(a\).

If \(u,w\) both have type \(B\), then both are adjacent to \(z\). Since the rim is an induced cycle of length at least five, the two rim neighbors of \(z\) are nonadjacent. Both list \(a\) and both have the neighbor \(z\) omitting \(a\). Hence \(u,w\) are nonadjacent \(a\)-special vertices, giving outcome (ii).

Let \(b\) have type \(B\) and let \(u\) have type \(X\). The vertex \(b\) sees \(z\), while \(u\) omits \(a\).

If \(bu\notin E(G)\), choose \(u\) as the vertex colored \(p\) instead of \(z\). Since \(u\) is \(x\)-only and omits \(a\), exactness of the auxiliary rule gives \(p\in L(u)\setminus M(u)\) and \(M(u)=S\). Assign \(u=p\) and alternate \(S\) on \(C_n-u\). The extension is proper, and it keeps \(x=a\); hence it is in \(\cF(x,a)\) and is equivalent to \(\varphi\) by \cref{lem:fiber}. The adjacent rim vertices \(b,z\) then receive opposite colors of \(S\). Thus \(b\) is the unique neighbor of \(x\) with its \(S\)-color and is not adjacent to the new vertex colored \(p\); the preceding two changes apply.

If \(bu\in E(G)\), then \(zb\) and \(bu\) are consecutive rim edges. On the induced vertex set \(\{x\}\cup V(C_n)\), the four branch vertices \(x,z,b,u\) have every \(K_4\) edge except \(zu\), and the other \(z,u\) rim arc subdivides precisely that missing edge. The arc has an internal vertex because \(n\ge5\). The rim is induced and \(x\) has no additional rim neighbors, so this is an induced \(K_4^+\), giving outcome (iv).

Now \(z,u,w\) all have type \(X\) and omit \(a\). Consider the three cyclic gaps between them, measured in rim edges.

If one gap is even, let \(r_0\) be the third type-\(X\) vertex and assign \(p\) to it instead of \(z\). Since \(r_0\) is \(x\)-only and omits \(a\), exactness gives \(p\in L(r_0)\setminus M(r_0)\) with \(M(r_0)=S\). Assign \(r_0=p\) and alternate \(S\) on \(C_n-r_0\). This is proper and remains in \(\cF(x,a)\), so \cref{lem:fiber} connects it to \(\varphi\). The other two type-\(X\) vertices then receive the same color of \(S\), leaving the other color free at \(x\); outcome (iii) follows.

Assume all three gaps are odd. If two gaps have length one, the three type-\(X\) vertices are consecutive and \(x\) together with the rim again induces a \(K_4^+\), the remaining arc subdividing the missing edge.

Otherwise at least two gaps have length at least three. Orient the cycle and take the first edge of each of two such gaps. Each selected edge joins a type-\(X\) vertex to a type-\(V\) vertex, because \(x\) has no other rim neighbor. With the boundary edges, these give two internally disjoint length-three \(x,v\) paths. Their four internal vertices are distinct and have no cross-edge: each chosen type-\(V\) vertex is separated from the next type-\(X\) vertex by the length of the selected gap, and the two selected initial type-\(X\) vertices are separated by one of the long gaps. Together with the edge \(xv\), the induced six-vertex subgraph is the bipartite theta with path lengths \(1,3,3\). This is outcome (iv).

\subsection{Rooted block colorings and neighborhood constraints}\label{sec:rooted-colorings}

We may now assume that \(H\) has more than one block. Let \(r\) be adjacent to \(x\) but not to \(v\), with \(a\notin L(r)\) and \(p\in L(r)\setminus M(r)\).

Initially take \(r=z\). Root the block-cut tree at \(r\). For each block \(B\), let \(e_B\) be its entry vertex toward \(r\) and put \(P_B=B-e_B\).

The blockwise-uniform assignment gives a palette \(S_B\) for each block. A clique block can be colored by assigning the colors of \(S_B\) bijectively to \(P_B\). If \(B\) is an odd cycle, the even-order path \(P_B\) can be colored alternately from its two-color palette, in either orientation. These choices are independent because palettes of blocks meeting at a cutvertex are disjoint. Giving \(r\) the external color \(p\) therefore produces a proper coloring of \(H\) from \(M\) except at \(r\); extending by \(v=p,x=a\) is proper.

Conversely, every proper coloring that uses \(M\) off \(r\) and assigns \(p\) to \(r\) has this form. To see this, repeatedly strip a leaf block not containing \(r\). In a clique, the internal vertices use every color of the block palette. In an odd cycle, the even-order interior path has the two palette colors at its ends. Thus the cutvertex cannot use a child-block color. Delete the interior and remove that palette contribution from the cutvertex. Induction proves that every nonentry vertex uses its parent-block palette and every entry uses a color outside each child palette. At the final root block, \(p\notin M(r)=\bigcup_{B\ni r}S_B\), so the external root color is outside every incident block palette; this closes the exhaustion argument at the root.

Every extension is \(L\)-equivalent to \(\varphi\) by \cref{lem:fiber}, since it assigns \(a\) to \(x\). The connecting sequence need not fix \(r\) or \(v\), or preserve the obstruction sets.

Put \(Q_x=L(x)\setminus\{p,a\}\), so \(|Q_x|=d_G(x)-2\).

If some rooted coloring leaves a color \(q\in Q_x\) absent from \(N(x)\), then the \((a,q)\)-component of \(x\) is the singleton \(\{x\}\); change \(x\) to \(q\). Every \(a\)-colored neighbor of \(v\) is then \(M\)-colored and, by the boundary table, lists \(p\), so the first obstruction set is empty. If a second-layer path \(v-y-w\) remains, then \(y\) cannot be adjacent to \(x\) because the \(x\)-boundary rule deletes \(a\). Thus \(x\) and \(y\) are nonadjacent \(a\)-special vertices. We obtain outcome (i) or (ii).

Assume that no rooted coloring has a free color of \(Q_x\). In every rooted coloring, the \(d_G(x)-1\) neighbors of \(x\) in \(H\) include \(r=p\) and \(d_G(x)-2\) other vertices. No \(M\)-colored \(x\)-neighbor can use \(a\). To cover all colors of \(Q_x\), those other vertices must use \(Q_x\) bijectively. Consequently, in every rooted coloring:

\begin{itemize}
\tightlist
\item
  \(r\) is the unique \(p\)-colored neighbor of \(x\) in \(H\);
\item
  the remaining \(x\)-neighbors receive pairwise distinct colors of \(Q_x\).
\end{itemize}

Call a block \(B\) \emph{active} if \(P_B\cap N(x)\ne\varnothing\). Fix the colorings of all other blocks. The multiset of colors on \(P_B\cap N(x)\) is invariant under every permissible choice in \(B\). For a clique, this subset is either empty or all of \(P_B\). For an odd cycle, the numbers of its vertices in the two parity classes are equal and each is at most one, by uniqueness at \(x\). Hence every active block contributes its full palette, contained in \(Q_x\), and distinct active palettes are disjoint.

\subsection{Endblock recoloring and completion}\label{sec:endblock-argument}

Let \(B\) be an endblock with cutvertex \(c\) and interior \(I=B-c\). Three-connectivity forces \(I\) to have a neighbor of \(v\) and a neighbor of \(x\); otherwise \(\{c,v\}\) or \(\{c,x\}\) separates \(I\). Call \(B\) \emph{separated} if no vertex of \(I\) sees both \(v\) and \(x\), and \emph{nonseparated} otherwise.

Suppose \(u\in I\) is adjacent to \(x\) and lists \(a\). We show that a recoloring gives outcome (i) or (ii), unless \(u\) is adjacent to both boundaries and has a \(p\)-colored neighbor \(g\) with \(a\notin L(g)\).

\proofpart{An endblock containing the root}

Suppose \(r\in B\); it may equal \(c\). Since \(r\) omits \(a\), we have \(u\ne r\), and \(B\) is active. Choose the rooted block coloring so that \(u\) receives a color \(q\in S_B\subseteq Q_x\). The uniqueness of each color at \(x\) makes \(u\) the unique \(q\)-colored neighbor of \(x\). Every neighbor of \(u\) in \(H\) is in \(B\) and none is colored \(a\). Hence the \((a,q)\)-component at \(x\) is the edge \(xu\). Swap it, obtaining \(x=q\) and \(u=a\).

If \(u\) sees \(v\), exactness at a common boundary vertex gives \(p\in L(u)\), so \(u\) is not a vertex of \(B_1\). Every other \(a\)-colored neighbor of \(v\) is still \(M\)-colored and is safe. A remaining second-layer path through a middle vertex different from \(u\) gives a nonadjacent common-\(a\) special pair with \(x\). A path through \(u\) can use only the unique \(p\)-vertex \(r\) in \(B\) besides \(v\). Therefore it requires \(uv,ur\in E(G)\) and \(a\notin L(r)\).

We retain this exceptional \(p\)-colored neighbor \(r\). In a separated endblock, \(uv\) is absent, so outcome (i) or (ii) follows.

\proofpart{An endblock not containing the root}

Now \(r\notin B\), and \(c\) receives its color from its parent block \(D\). Choose \(u=q\in S_B\subseteq Q_x\).

\begin{itemize}
\tightlist
\item
  If the parent palette \(S_D\) contains a color outside \(\{a,p\}\), give \(c\) such a color.
\item
  If \(S_D=\{a,p\}\), give \(c\) the color \(p\); this is safe because then \(a\in L(c)\).
\item
  If \(S_D=\{p\}\), give \(c\) the color \(p\); this is safe unless \(a\notin L(c)\).
\item
  If \(S_D=\{a\}\), then \(c=a\) is forced; this is the alternating-path case below.
\end{itemize}

Whenever \(c\) is not colored \(a\), or \(uc\notin E(G)\), the \((a,q)\)-component at \(x\) is exactly the edge \(xu\). If \(c=q\), properness itself forces \(uc\notin E(G)\). After swapping \(xu\), the same first-layer verification applies. A remaining vertex of \(B_2\) reached through \(u\) requires all of the following:

\[
u\text{ sees both }v,x,\quad uc\in E(G),\quad \varphi(c)=p,
\quad a\notin L(c),\quad S_D=\{p\}.
\]

We retain this exceptional \(p\)-colored neighbor \(c\).

\proofpart{A forced alternating path}

Suppose \(uc\in E(G)\) and \(S_D=\{a\}\). Then \(D\) is the edge \(cc_1\) and \(c=a\). At \(c_1\), choose a parent-block color different from \(q\) whenever possible. If it is forced to \(q\), its parent palette is the singleton \(\{q\}\) and the parent block is another edge. Continue toward \(r\), alternating singleton palettes \(\{a\},\{q\}\) whenever no alternative exists. The process terminates because \(r=p\notin\{a,q\}\).

The \((a,q)\)-component at \(x\) is exactly the path beginning \(x-u-c\) and continuing along those forced edges up to, but not including, the first parent vertex whose color is outside \(\{a,q\}\). To check this, note that \(S_B\) omits \(a\), and \(u\) is the only \(q\)-colored neighbor of \(c\) in \(B\): for a clique, each palette color occurs once on \(B-c\); for an odd cycle, the ends of \(B-c\) have opposite colors. At \(c\), every side-block palette is disjoint from \(S_B\) and \(\{a\}\), and at each subsequent included vertex it is disjoint from the incident singleton palettes \(\{a\}\) and \(\{q\}\). Thus there is no side exit in \(H\). The uniqueness of each color at \(x\) prevents another \(q\)-colored neighbor of \(x\), properness prevents an \(a\)-colored neighbor of \(x\), and the other boundary vertex \(v\) has color \(p\notin\{a,q\}\). Hence the path is a whole \((a,q)\)-component of \(G\).

The list conditions are checked separately. At \(x\), both colors belong to \(L(x)\), since \(\varphi(x)=a\) and \(q\in Q_x\subseteq L(x)\). At \(u\), we have \(q\in S_B\subseteq M(u)\subseteq L(u)\), while \(a\in L(u)\) is the hypothesis of this recoloring argument. At every remaining path vertex, the two colors belong to its incident block palettes and hence to its original list. The swap on this whole component is therefore \(L\)-valid.

Swap it. Then \(x=q\) and \(u=a\). Every other newly \(a\)-colored vertex was formerly \(M\)-colored \(q\) and lists \(a\) through its incident palette. If such a vertex sees \(v\), the boundary table forces \(p\) into its original list; if it saw \(x\), it would contradict uniqueness of the \(q\)-colored neighbor. Hence the first layer is empty. Any second-layer path through a middle vertex different from \(u\) gives outcome (ii). Through \(u\), no \(p\)-neighbor remains in \(H\): the entry \(c\) now has color \(q\), and the other block neighbors use colors outside \(\{p,a\}\). Therefore outcome (i) or (ii) always follows.

The preceding argument gives the following alternatives:

\begin{itemize}
\tightlist
\item
  an unresolved separated endblock has every interior \(x\)-neighbor omitting \(a\);
\item
  an unresolved nonseparated endblock supplies an interior vertex that is \(a\)-special, adjacent to both \(v,x\), with a fixed \(p\)-colored neighbor omitting \(a\).
\end{itemize}

It remains to treat the exceptional case of a nonroot endblock \(B\) of order at least three with cutvertex \(c\) and parent edge \(D=cd\) of palette \(\{p\}\). Thus \(\varphi(c)=p\) and \(a\notin L(c)\), and some \(u\in B-c\) is adjacent to \(c,v,x\) and satisfies \(a\in L(u)\). We show that side branches at \(c\) cause no obstruction.

First, \(c\) sees neither boundary vertex. It cannot see \(x\), because it would be a second \(p\)-colored neighbor of \(x\) besides \(r\). It cannot see \(v\), because \(p\in M(c)\) through the parent block, whereas the boundary table removes \(p\) at every \(v\)-neighbor.

Let \(S=S_B\). Then \(|S|\ge2\) and \(S\subseteq Q_x\). The parent edge cannot contain \(r\), because its palette \(\{p\}\) would then be contained in \(M(r)\), contrary to \(p\notin M(r)\). Hence \(d\) is a nonentry vertex of a further parent block \(E\) with nonempty palette \(R\). First choose \(\tau\in R\), and then choose \(s\in S\setminus\{\tau\}\). This choice is possible because \(|S|\ge2\). Keep the parent-block coloring with \(d=\tau\) fixed, and independently color the terminal block so that \(u=s\). The color of \(c\) remains \(p\).

The vertex \(u\) permits \(a\), as proved in the preceding recoloring argument. The whole \((a,s)\)-component of \(x\) is exactly the edge \(xu\). At \(x\), the vertex \(u\) is its unique \(s\)-colored neighbor. At \(u\), the boundary \(v\) and the entry \(c\) are both colored \(p\), while all other endblock neighbors avoid \(a,s\). Swap \(xu\), giving \(x=s,u=a\).

Now \(c\) has no \(s\)-colored neighbor. The parent neighbor \(d\) has color \(\tau\ne s\). On the \(B\) side, the unique \(s\)-neighbor \(u\) has changed to \(a\): in a clique the nonentry colors are distinct, and in an odd cycle \(u\) is one endpoint of the even-order path and the other endpoint has the opposite palette color. Every child side-block at \(c\) has a palette disjoint from \(S_B\), so none contributes an \(s\)-colored neighbor. There is no boundary neighbor at \(c\), and properness excludes \(p\) from \(N(c)\). Therefore the whole \((p,s)\)-component of \(c\) is the singleton \(\{c\}\). Since \(p,s\in M(c)\), change \(c\) from \(p\) to \(s\).

The first obstruction set is still empty: the only new \(a\)-vertex is the common boundary vertex \(u\), and \(p\in L(u)\). A vertex of \(B_2\) through \(u\) is impossible because its only remaining \(p\)-neighbor is \(v\). Every other \(a\)-colored middle vertex is \(M\)-colored and nonadjacent to \(x\), and hence gives outcome (ii) with \(x\). Thus we obtain outcome (i) or (ii).

If a recoloring at an endblock gives outcome (i) or (ii), we are done. Otherwise every nonseparated endblock contains a common-boundary \(a\)-special vertex. Two nonseparated endblocks give two nonadjacent such vertices. Indeed, an interior vertex of an endblock belongs to exactly one block of \(H\); if an edge joined interiors of two distinct endblocks, the unique block containing that edge would have to be both endblocks. Thus distinct endblock interiors are anticomplete, and outcome (ii) follows.

If there are two separated endblocks, each contains an induced length-three \(v,x\) path through its interior. In a clique endblock, choose an interior \(v\)-only vertex and an interior \(x\)-only vertex; both types occur and are adjacent. In an odd-cycle endblock, the interior is a path whose vertices each have exactly one boundary label, and an adjacent change of label supplies the required path. The preceding anticompleteness proof excludes every cross-edge between the two interiors. Together with the edge \(vx\), the two paths induce the bipartite theta with path lengths \(1,3,3\). This is outcome (iv).

A connected multiblock graph has at least two endblocks. If it had at least three, two would have the same separated/nonseparated type and one of the preceding exits would apply. Hence the only unresolved pattern has exactly two endblocks, one separated \(B_S\) and one nonseparated \(B_C\). The leaves of the block-cut tree are precisely the endblocks, and every cutvertex node has degree at least two. A finite tree with exactly two leaves has maximum degree at most two, so the block-cut tree is a path.

Choose an interior \(x\)-neighbor \(r\) in \(B_S\). It omits \(a\) by the preceding recoloring argument and is \(x\)-only by separation. Exactness gives \(p\in L(r)\setminus M(r).\)

Root the same graph \(H\), with the same assignment \(M\) and block palettes, at this new vertex \(r\). Give \(r\) color \(p\) and choose the block colorings as in \cref{sec:rooted-colorings}. Each proper extension with \(v=p,x=a\) is equivalent to \(\varphi\) by \cref{lem:fiber}. Apply the counting argument of \cref{sec:rooted-colorings} to the new rooted family. A free color gives outcome (i) or (ii); otherwise the count shows that \(r\) is the unique \(p\)-colored neighbor of \(x\) in \(H\), and every active block contributes its full palette.

Now \(B_C\) is nonroot. It gives outcome (i) or (ii), or supplies a nonroot entry \(c\) of color \(p\) whose parent palette is \(\{p\}\). In the latter case, \(c\) sees neither boundary. If \(B_C\) were an edge, then \(c\) would have only its terminal and parent edges in the block-cut path and no boundary neighbor, giving \(d_G(c)=2\), impossible in a 3-connected graph. Hence \(|B_C|\ge3\), and the recoloring argument of \cref{sec:endblock-argument} applies. This proves \cref{prop:reduction}.\hfill\(\square\)

\begin{corollary}\label{cor:absorb}
Under the assumptions of \cref{prop:reduction}, suppose also that $a$ is special for $v$. Then either $\varphi$ is equivalent to a coloring in $\cF(v,a)$, or $G$ has two nonadjacent special vertices with common special color $a$.
\end{corollary}
\begin{proof}
By \cref{lem:fiber}, the set $\cF(v,a)$ is nonempty. Outcome (iii) already gives the first conclusion, and outcome (iv) makes all $L$-colorings equivalent by \cref{lem:kernels}. In outcome (i), let $\psi$ be the coloring with both obstruction sets empty and let $K=K_\psi(v;p,a)$. If swapping $K$ is valid, it sends $v$ to $a$. Otherwise choose an invalid vertex $w$ of minimum distance from $v$ in $K$. Its distance is at least three. Let $y$ precede $w$ on a shortest path, and let $\beta=\psi(y)$. Then $\beta\in L(y)\setminus L(w)$, so $\beta$ is special for $y$. Also $vy\notin E(G)$: since $v,y\in K$, an edge between them would give distance one. By \cref{lem:pair}, the coloring $\psi\in\cF(y,\beta)$ is equivalent to a coloring in $\cF(v,a)$. Only outcome (ii) can remain.
\end{proof}

\section{Vertices with a common special color}\label{sec:common}

\begin{theorem}\label{thm:common}
Let $G$ be $3$-connected and let $L$ be a finite degree-assignment. If $x,y$ are nonadjacent special vertices and $c$ is special for both, then all $L$-colorings of $G$ are $L$-equivalent.
\end{theorem}

\subsection{A minimal obstruction and reduced lists}\label{sec:minimum-lists}

If \(L\) has strict surplus, \cref{lem:surplus} proves the theorem. Assume henceforth that \(L\) is tight.

By \cref{lem:pair}, the nonempty set \(\cF=\cF(x,c)\cup\cF(y,c)\) is contained in one Kempe class. Suppose there is another class \(\cQ\), so that \(\cQ\cap\cF=\varnothing\).

Fix \(\varphi_0\in\cQ\) and put \(p=\varphi_0(x)\ne c\). For any coloring \(\varphi\) with \(\varphi(x)=p\), let \(K_\varphi\) be the whole \((p,c)\)-component of \(x\). Define \(B_i(\varphi)\), for \(i=1,2\), using distance in \(K_\varphi\) and the fixed-list validity of the proposed swap, as in \cref{sec:preliminaries}.

We may choose a coloring in \(\cQ\) with \(x=p\) and \(B_1=\varnothing\). To see this, suppose \(B_1(\varphi_0)\ne\varnothing\), and define \(L'\) on \(G-x\) by deleting \(p\) from a neighbor list when possible and deleting \(c\) otherwise; leave nonneighbor lists unchanged.

The graph \(G-x\) is 2-connected. If it were not \(L'\)-colorable, \cref{lem:gallai} would make it a clique or an odd cycle with one common reduced list.

The clique case is impossible. The nonneighbor \(y\) retains \(c\) in \(L'(y)\), while a special neighbor of \(y\) that omits \(c\) is also in \(G-x\) and omits \(c\) from its reduced list. Thus the reduced lists are not identical.

The odd-cycle case is also impossible. Every cycle vertex has degree at least three in \(G\) and has only \(x\) outside the cycle, so every cycle vertex would be adjacent to \(x\), contrary to \(xy\notin E(G)\).

Hence \(G-x\) has an \(L'\)-coloring. Extend it by \(x=p\). The extension is proper and has empty first obstruction set. The original nonempty first layer witnesses that \(p\) is special for \(x\), so \cref{lem:fiber} connects the old and new \(p\)-colored endpoints. We retain the notation \(\varphi\) for the new coloring.

If a coloring in \(\cQ\) has \(x=p\) and both \(B_1,B_2\) empty, then it reaches \(\cF\). Indeed, a valid swap of \(K_\varphi\) sends \(x\) to \(c\). If the swap is invalid, choose a nearest invalid vertex \(t\); its component distance is at least three. Let \(z\) precede \(t\) and let \(\beta=\varphi(z)\). Then \(\beta\in L(z)\setminus L(t)\), so \(\beta\) is special for \(z\), and \(xz\notin E(G)\). \Cref{lem:pair} connects the current coloring in \(\cF(z,\beta)\) to \(\cF(x,c)\), a contradiction.

Therefore every coloring in

\[
\mathcal D=
\{\varphi\in\cQ:\varphi(x)=p,\ B_1(\varphi)=\varnothing\}
\]

has \(B_2(\varphi)\ne\varnothing\). Choose \(\varphi\in\mathcal D\) minimizing \(\mu(\varphi)=|B_2(\varphi)|>0.\)

Select \(t\in B_2(\varphi)\) and a shortest component path \(x-z-t\). Then

\[
\varphi(x)=\varphi(t)=p,
\qquad
\varphi(z)=c,
\qquad
c\notin L(t),
\qquad
xt\notin E(G).
\]

The vertices \(x,z,t,y\) are distinct. The vertex \(z\) is \(c\)-special, witnessed by \(t\). If \(zy\notin E(G)\), \cref{lem:pair} applied to \(z,y\) connects the present coloring to \(\cF(y,c)\). Hence \(xz,zt,zy\in E(G)\) and \(xy,xt\notin E(G)\). No assumption has been made about the edge \(ty\).

Put \(H=G-\{x,t\}\) and \(J=G-\{x,t,y\}\). Define \(L_H\) by deleting \(p\) from the list of each neighbor of \(x\) or \(t\), and leaving every other list unchanged. A vertex adjacent to both loses \(p\) only once, so \(L_H\) is a degree-assignment. The vertex \(z\) loses one color but two neighbors, and therefore \(|L_H(z)|=d_G(z)-1>d_H(z).\)

The restriction \(\varphi|_H\) is one such \(L_H\)-coloring: properness at the fixed boundary vertices \(x=t=p\) implies that no neighbor of \(x\) or \(t\) is colored \(p\), so the restriction respects every deletion used to define \(L_H\). By \cref{lem:surplus}, all \(L_H\)-colorings of \(H\) are equivalent.

For each component \(C\) of \(J\), define \(M_C\) by additionally deleting \(c\) at every neighbor of \(y\) in \(C\). Again this is a degree-assignment. If every component were \(M_C\)-colorable, color them independently and assign \(c\) to \(y\). This gives an \(L_H\)-coloring of \(H\). The equivalence in \(H\) lifts to \(G\) with \(x=t=p\) fixed, as proved in \cref{sec:terminal-recolorings}, and reaches \(\cF(y,c)\), a contradiction.

Thus some component \(C\) is not \(M_C\)-colorable. \Cref{lem:gallai} makes every uncolorable component a blockwise-uniform Gallai tree, with \(|M_C(u)|=d_C(u)\) for every \(u\in C\).

Equality in the deletion bounds gives the following properties.

\begin{enumerate}
\def\labelenumi{\arabic{enumi}.}
\tightlist
\item
  No vertex of \(C\) is adjacent to both \(x\) and \(t\).
\item
  Every neighbor of \(x\) or \(t\) in \(C\) lists \(p\), and every neighbor of \(y\) in \(C\) lists \(c\).
\item
  With \(M=M_C\),
  \[
  L(u)=M(u)\mathbin{\dot\cup}
  \bigl(\{p\}\text{ if }u\text{ sees }x\text{ or }t\bigr)
  \mathbin{\dot\cup}
  \bigl(\{c\}\text{ if }u\text{ sees }y\bigr).
  \]
\item
  The component \(C\) has a neighbor to each of \(x,t,y\); otherwise its external neighborhood would have size at most two.
\item
  The component of \(J\) containing \(z\) is colorable from its reduced lists, because \(z\) is adjacent to both \(x,t\) and has strict surplus after the reduction.
\end{enumerate}

Consequently \(z\) is anticomplete to every uncolorable component, and an uncolorable component cannot be a singleton.

\subsection{Structure of the uncolorable component}\label{sec:component-structure}

\proofpart{Special vertices in the uncolorable components}

Let \(v\in C\) be special, with a special color \(a\). Since \(z\) and \(v\) are in different components of \(J\), they are nonadjacent. \Cref{lem:pair} connects the current coloring in \(\cF(z,c)\) to \(\cF(v,a)\). If \(v\) is nonadjacent to either \(x\) or \(y\), a second application connects to \(\cF(x,c)\) or \(\cF(y,c)\). Therefore every special vertex in an uncolorable component must be adjacent to both \(x\) and \(y\). Call a vertex of an uncolorable component an \emph{\(XY\) vertex} if it is adjacent to both \(x\) and \(y\).

There is at most one \(XY\) vertex over all uncolorable components. If \(r_1,r_2\) were two, then on the induced set \(\{x,y,z,r_1,r_2\}\):

\begin{itemize}
\tightlist
\item
  if \(r_1r_2\notin E(G)\), the three length-two \(x,y\) paths through \(z,r_1,r_2\) form a bipartite theta;
\item
  if \(r_1r_2\in E(G)\), the four vertices \(x,y,r_1,r_2\) form \(K_4\) minus \(xy\), while \(x-z-y\) subdivides the missing edge, giving \(K_4^+\).
\end{itemize}

There are no other edges: \(xy\) is absent and \(z\) is anticomplete to uncolorable components. \Cref{lem:kernels} would make \(G\) degree-swappable, contradicting \(\cQ\).

For an uncolorable component put \(S_C=\{u\in C:c\in L(u)\}.\)

No neighbor \(u\) of \(t\) in \(C\) is in \(S_C\). Otherwise \(u\) is \(c\)-special, witnessed by \(t\). Since \(z\) is \(c\)-special and anticomplete to \(C\), \cref{lem:pair} first connects the present coloring in \(\cF(z,c)\) to \(\cF(u,c)\). The boundary equality makes \(u\) nonadjacent to \(x\), so a second application connects to \(\cF(x,c)\), a contradiction.

A \(y\)-neighbor is in \(S_C\), while a \(t\)-neighbor is not. Along a path between them, an edge crosses \(S_C\) and its complement. The inside endpoint is \(c\)-special and hence must be an \(XY\) vertex. Thus every uncolorable component has an \(XY\) vertex. Uniqueness gives exactly one uncolorable component \(C\) and exactly one such vertex \(r\). Every edge between \(S_C\) and \(C-S_C\) has its endpoint in \(S_C\) at \(r\).

The vertex \(r\) is \(c\)-special. It is nonadjacent to \(z\) and to \(t\); the latter follows because \(r\) sees \(x\) and no vertex of \(C\) sees both \(x,t\). Every vertex of \(C-r\) is nonspecial.

\proofpart{The vertex $t$ is nonspecial}

Suppose \(a\) were special for \(t\). Consider the following sets of colorings

\[
\cF(z,c)
\longleftrightarrow
\cF(r,c)
\longleftrightarrow
\cF(t,a)
\longleftrightarrow
\cF(x,c).
\]

The consecutive vertex pairs are nonadjacent: \(z,r\) are in different \(J\)-components, \(r\) is nonadjacent to \(t\), and \(x,t\) are nonadjacent because both are colored \(p\). Each arrow is \cref{lem:pair}. The sequence takes the current coloring from \(\cQ\) to the target class, a contradiction. Hence \(t\) is nonspecial. This argument does not depend on whether \(ty\) is an edge.

\proofpart{Only $r$ lists $c$}

Every component \(D\) of \(C-r\) is either contained in \(S_C\) or disjoint from \(S_C\), because every edge between \(S_C\) and its complement has endpoint \(r\) in \(S_C\). Suppose \(D\subseteq S_C\). It has no neighbor to \(t\).

If \(C\) has one block, then \(C-r\) is connected and contains a \(t\)-neighbor, a contradiction. If \(C\) has more than one block, choose a farthest endblock \(B\) in the branch containing \(D\), with interior \(I\) avoiding \(r\) and contained in \(D\). Three-connectivity forces \(I\) to attach to at least two of \(x,t,y\). It misses \(t\), so it attaches to \(x\) and \(y\). The two attachments are distinct: no vertex of \(I\) is adjacent to both \(x\) and \(y\), since \(r\) is the unique \(XY\) vertex.

All vertices of \(I\) have the same reduced block palette \(S_B\). At an \(x\)-neighbor the original list contains \(c\) and the reduction deletes \(p\), so \(c\in S_B\). At a \(y\)-neighbor the reduction deletes \(c\), so \(c\notin S_B\), a contradiction. Hence every component of \(C-r\) is disjoint from \(S_C\). Thus \(S_C=\{r\}\), and every \(y\)-neighbor must be \(r\).

Every component \(D\) of \(C-r\) has all external neighbors in \(\{r,x,t\}\). Three-connectivity forces an attachment to both \(x\) and \(t\). All vertices of \(D\) are nonspecial, so adjacent vertices have equal lists and all vertices of \(D\) share one list \(T_D\).

At a \(t\)-attachment \(s\), nonspecialness of \(s\) gives \(T_D\subseteq L(t)\), while nonspecialness of \(t\) gives \(L(t)\subseteq T_D\). Hence every \(T_D\) is the same list \(T=L(t).\)

The same inclusions at attachments to \(x\) and \(r\) give \(T\subseteq L(x),L(r)\). Nonspecialness of \(t\) gives \(T\subseteq L(z)\), and also \(T\subseteq L(y)\) if \(ty\) is present. We have \(p\in T\) and \(c\notin T\). Tightness gives \(d_G(u)=d_G(t)=|T|=:k\) for every \(u\in C-r\). Since \(G\) is 3-connected, \(k\ge3\).

\proofpart{Terminal blocks force $k=3$}

First suppose \(C\) has more than one block. Choose an endblock \(B\) whose interior \(I\) avoids \(r\). It has no \(y\)-neighbor, so its interior must contain both neighbors of \(x\) and neighbors of \(t\). At any such neighbor the reduced list is \(T\setminus\{p\}\), and therefore the block palette is \(S_B=T\setminus\{p\}.\)

A vertex of \(I\) adjacent to neither \(x\) nor \(t\) would have reduced list \(T\), a contradiction. Thus every vertex of \(I\) is adjacent to exactly one of \(x,t\), and both types occur.

If \(B\) is an odd cycle, then \(|S_B|=2\) and \(k=3\). If \(B=K_m\), then \(|S_B|=m-1\), so \(m=k\). An edge block has only one interior vertex, so it cannot contain both types. If \(k\ge4\), among the at least three interior vertices choose two of one type and one of the other. If the repeated type is the \(x\)-type, then \(\{x,z,t,a_1,a_2,s\}\) induces a \(K_4^+\) whose branch vertices are \(x,a_1,a_2,s\) and whose missing \(xs\) edge is replaced by \(x-z-t-s\). The symmetric construction uses the path \(t-z-x-s\) when the repeated type is the \(t\)-type. Each chosen interior vertex is adjacent to exactly one of \(x,t\), and \(z\) is anticomplete to \(C\), so the subgraph is induced; the vertex \(y\) is not selected, so a possible edge \(ty\) is irrelevant. This contradicts \(\cQ\). Therefore \(k=3\), and a terminal clique is a triangle.

Now suppose \(C\) has one block. Since it attaches to \(t\), the common block palette is again \(T\setminus\{p\}\). Every vertex other than \(r\) must be adjacent to exactly one of \(x,t\); otherwise its reduced list would be \(T\), not \(T\setminus\{p\}\). If the block is an odd cycle, then \(k=3\).

Suppose instead that \(C=K_m\); then \(m=k\ge3\). If \(C-r\) contains a neighbor \(u\) of \(x\) and a neighbor \(s\) of \(t\), then \(\{x,z,t,r,u,s\}\) induces a \(K_4^+\), with missing edge \(xs\) replaced by \(x-z-t-s\). If all \(k-1\) vertices other than \(r\) are neighbors of \(t\), then those vertices together with \(z\) exhaust all \(k\) neighbors of \(t\). Hence \(ty\) is absent as a consequence of the degree count. Taking two such neighbors \(s_1,s_2\), the induced set \(\{t,z,y,r,s_1,s_2\}\) is a \(K_4^+\) whose missing edge \(tr\) is replaced by \(t-z-y-r\); the just-proved absence of \(ty\) excludes a chord. Both alternatives contradict \(\cQ\). Thus a one-block clique is impossible.

Write \(A=T\setminus\{p\}=\{a,b\}\). Thus \(p,a,b,c\) are distinct. Every block palette incident with \(r\) is in \(T\), while \(p\notin M(r)\) because \(r\) sees \(x\). Hence \(M(r)\subseteq A\). Conversely \(T\subseteq L(r)\) and the exact boundary decomposition gives \(L(r)=M(r)\mathbin{\dot\cup}\{p,c\}.\)

Thus \(A\subseteq M(r)\) and consequently

\[
M(r)=A,
\qquad
L(r)=T\cup\{c\},
\qquad
d_C(r)=2,
\qquad
d_G(r)=4.
\]

\subsection{Terminal recolorings and strict descent}\label{sec:terminal-recolorings}

On \(C\), the current \(L_H\)-assignment is exactly \(M\) with \(c\) restored only at \(r\). Every such coloring uses \(c\) at \(r\); otherwise it would be an \(M\)-coloring of the uncolorable component. No other vertex of \(C\) can use \(c\).

Root the blocks of \(C\) at \(r\). The leaf-block argument in \cref{sec:rooted-colorings} shows that every reduced-list coloring is obtained by assigning \(c\) to \(r\), permuting the palette \(S_B\) on the nonentry vertices of each clique block, and alternating its two colors on \(B-e_B\) for each odd-cycle block. The entry color is outside \(S_B\). Conversely, all these choices may be made independently, since palettes at a cutvertex are disjoint. In particular, the orientation of an endblock path may be changed while all other vertices stay fixed.

Insert any desired rooted coloring of \(C\) into \(H\), leaving every vertex outside \(C\) fixed, including \(y\) with color \(d=\varphi(y)\). This is an \(L_H\)-coloring: the only \(C\)-neighbor of \(y\) is \(r\), which remains colored \(c\ne d\). Since \(z\) has strict surplus, \cref{lem:surplus} connects the old and new \(L_H\)-colorings.

Every \(L_H\)-valid Kempe change in \(H\) is also an \(L\)-valid Kempe change in \(G\), with \(x=t=p\) fixed. If the two colors omit \(p\), the changed component cannot meet the fixed boundary vertices. If the pair contains \(p\), every neighbor of \(x\) or \(t\) has \(p\) deleted from its \(L_H\)-list. Such a vertex cannot be in a valid swapped component: if it currently had the other color, the swap would assign forbidden \(p\). Hence the component has no corresponding bichromatic edge to \(x\) or \(t\) and is whole in \(G\).

Only the endpoints are used here; the obstruction sets need not remain fixed along the lifted sequence. Every coloring of the sequence is still in \(\cQ\).

\proofpart{Interchanging the two colors on a terminal block}

If \(C\) has more than one block, choose an endblock \(B\) whose interior avoids \(r\), and let \(q\) be its cutvertex. If \(C\) has one block, take \(B=C\) and \(q=r\). In either case, call \(B\) a \emph{terminal block} and let \(P=B-q\).

The structural argument above shows that \(B\) is a triangle or an odd cycle with palette \(A=\{a,b\}\). Hence \(P\) is a path of even order, alternately colored from \(A\). Within \(C\), the only possible neighbor of \(P\) outside \(P\) is \(q\): this is immediate when \(B=C\), and otherwise follows because the interior vertices of an endblock belong to no other block. Since \(C\) is a component of \(G-\{x,t,y\}\), its other external neighbors are in \(\{x,t,y\}\). The only \(y\)-neighbor in \(C\) is \(r\), which is not in \(P\). Thus
\[
N_G(P)\setminus P\subseteq\{q,x,t\}.
\]
The entry rule gives \(\varphi(q)\notin A\), and \(\varphi(x)=\varphi(t)=p\notin A\). Therefore \(P\) is a component of \(G_\varphi(a,b)\), not just of its restriction to \(C\). Every vertex of \(P\) has list \(T=\{p,a,b\}\), so interchanging \(a\) and \(b\) on \(P\) is \(L\)-valid.

This change reverses the alternating orientation of \(P\) and fixes every vertex outside \(P\). Since \(A\cap\{p,c\}=\varnothing\), it leaves the \((p,c)\)-subgraph, its component distances, and \(B_1,B_2\) unchanged. Either orientation is therefore available within \(\mathcal D\), with the same value of \(\mu\).

\Needspace{10\baselineskip}
We shall use the following criterion. Suppose a sequence ends at a coloring for which:

\begin{enumerate}
\def\labelenumi{\arabic{enumi}.}
\tightlist
\item
  \(x\) remains colored \(p\);
\item
  no new vertex is colored \(c\);
\item
  every newly \(p\)-colored vertex other than \(x\) has no \(c\)-neighbor;
\item
  the selected vertex \(t\in B_2\) leaves the color set \(\{p,c\}\);
\item
  every other changed vertex of the old \((p,c)\)-component also leaves \(\{p,c\}\).
\end{enumerate}

Then the new \((p,c)\)-component \(K'\) of \(x\) is a proper subset of the old component \(K\). Newly \(p\)-colored vertices are isolated in the \((p,c)\)-subgraph and no new \(c\)-vertices exist, so only deletions can affect the component of \(x\).

Every vertex in the new first obstruction set was already in the old one. If \(u\) belongs to the new second obstruction set, take a path \(x-w-u\) in \(K'\). The vertex \(w\) was already colored \(c\), since no new vertex receives \(c\). Also, \(u\) was already colored \(p\): a newly \(p\)-colored vertex has no \(c\)-neighbor. Thus the same colored path existed before the changes. Properness rules out the edge \(xu\), so the old component distance was exactly two. The lists have not changed, and hence \(u\) belonged to the old second obstruction set. Therefore

\[
B_1(\text{new})\subseteq B_1(\text{old}),
\qquad
B_2(\text{new})\subsetneq B_2(\text{old}),
\]

where strictness follows because \(t\) is removed. Thus an empty first layer remains empty and \(\mu\) strictly decreases.

\proofpart{The case $ty\in E(G)$}

The vertex \(t\) has degree three and already sees \(z,y\), so \(N_G(t)=\{z,y,s\}\) for a unique vertex \(s\in C\). Choose the terminal block as above. Its interior contains a neighbor of \(t\); this must be the unique vertex \(s\). Let \(q\) be the entry and \(P=B-q\).

Put \(d=\varphi(y)\). If \(d\in A\) and the current orientation does not give \(\varphi(s)=d\), flip the whole component \(P\). If \(d\notin A\), keep either orientation. The terminal flip preserves the \((p,c)\)-component data and \(\mu\).

Now some color \(h\in A\) is absent from \(N(t)\): the vertex \(z\) has color \(c\notin A\), and the only possible \(A\)-colors at the other two neighbors have been made equal when necessary. Properness also excludes \(p\) from \(N(t)\). Hence the \((p,h)\)-component at \(t\) is the singleton \(\{t\}\). Change \(t\) from \(p\) to \(h\).

No vertex is newly colored \(p\) or \(c\), and \(t\) leaves the old \((p,c)\)-component. The descent criterion above applies, giving a coloring in \(\mathcal D\) with smaller \(\mu\), a contradiction.

\proofpart{The case $ty\notin E(G)$}

Now \(t\) has exactly two neighbors besides \(z\). Choose a terminal block \(B\) and its path \(P=B-q\). The path \(P\) contains one or both of these neighbors.

\Needspace{4\baselineskip}\noindent\emph{Only one neighbor of \(t\) is in \(P\).}\quad Let \(s\) be the unique neighbor of \(t\) in \(P\), and let \(w\) be the other neighbor of \(t\) besides \(z,s\). If \(\varphi(w)\in A\) and \(\varphi(s)\ne\varphi(w)\), flip \(P\) to make the two colors equal; otherwise keep the current orientation. Then some \(h\in A\) is absent from \(N(t)\). The singleton change \(t:p\to h\) satisfies the descent criterion above and contradicts minimality.

\Needspace{4\baselineskip}\noindent\emph{Both remaining neighbors of \(t\) are in \(P\).}\quad Let \(s_1,s_2\) be the two neighbors of \(t\) in \(P\). If their distance on \(P\) is even, they have the same color. The other color of \(A\) is absent from \(N(t)\), so the same singleton change applies.

We may therefore assume that \(s_1,s_2\) are at odd distance on \(P\) and have different colors.

Suppose one of \(s_1,s_2\), say \(s\), is internal to \(P\), and let \(h=\varphi(s)\). Its two path neighbors and the other vertex of \(\{s_1,s_2\}\) use the other color of \(A\). Since \(d_G(s)=3\), the \((p,h)\)-component through \(t\) is exactly the edge \(ts\). Swap it. The new \(p\)-vertex \(s\) has no \(c\)-neighbor, and the descent criterion above applies.

It remains to consider the case where \(s_1,s_2\) are the endpoints of \(P\). Suppose first that \(q\ne r\). Then \(q\) is the cutvertex of the endblock \(B\). Since \(d_G(q)=3\), its two terminal-block edges leave exactly one parent-side edge \(qw\) and no boundary attachment. Its original list is \(T\). The terminal palette \(A\) forces the remaining incident palette to be \(\{p\}\), so the parent block is the edge \(qw\). The rooted entry rule gives \(\varphi(q)=p\) and \(\varphi(w)\ne p\). Choose \(h\in A\) after fixing the actual color of \(w\), so that \(h\ne\varphi(w)\), and let \(s_h\) be the endpoint of \(P\) colored \(h\). The \((p,h)\)-component is exactly the path \(t-s_h-q.\)

Indeed, the other endpoint and the relevant path neighbors use the other color of \(A\), the parent neighbor \(w\) does not use \(h\), and \(z\) uses \(c\). Swap this component. The only new \(p\)-vertex is \(s_h\); its neighbors after the swap have colors \(h,h\) and the other member of \(A\), so it has no \(c\)-neighbor. The descent criterion above applies.

Finally, suppose \(q=r\). If there were another block, the rooted block tree would have a terminal block away from \(r\), whose interior would contain a third neighbor of \(t\) in \(C\). This is impossible because \(t\) has only two neighbors other than \(z\). Thus \(C=B\) is one odd-cycle block.

Choose \(h\in A\) different from \(d=\varphi(y)\), and let \(s_h\) be the endpoint colored \(h\). Since \(\varphi(r)=c\), the \((p,h)\)-component through \(t\) is the edge \(t s_h\). Swap it. At this intermediate coloring, the new \(p\)-vertex \(s_h\) is adjacent to the \(c\)-colored vertex \(r\), so strict descent does not yet follow.

Let \(h'\) denote the unique color in \(A\setminus\{h\}\). The vertex \(r\) has exactly four neighbors: \(x,y\) and the two cycle endpoints. Their colors are now \(p,d,p,h'\), so none has color \(h\). Since \(h\in L(r)\), the \((c,h)\)-component at \(r\) is the singleton \(\{r\}\). Change \(r\) from \(c\) to \(h\).

Only after this second change do the conditions of the descent criterion hold: no new \(c\) is created, the new \(p\)-vertex \(s_h\) has no \(c\)-neighbor, and the vertices \(t,r\) of the old \((p,c)\)-component have both left \(\{p,c\}\). The descent criterion now gives strict descent.

In each case we obtain a coloring in \(\mathcal D\) with fewer vertices in \(B_2\), contradicting the choice of \(\varphi\). Hence no Kempe class is disjoint from \(\cF(x,c)\cup\cF(y,c)\). By \cref{lem:pair}, this union is contained in one class, so all \(L\)-colorings are equivalent. This proves \cref{thm:common}.\hfill\(\square\)

\section{Proof of the main theorem}\label{sec:main-proof}
\begin{proof}[Proof of \cref{thm:main}]
Fix a degree-assignment $L$ on $G$.

If some vertex has strict surplus, \cref{lem:surplus} proves the result. Hence assume that $L$ is tight.

Suppose first that all lists are the same set $S$. Then $d_G(v)=|S|$ for every vertex, so $G$ is regular of degree at least three. \Cref{thm:regular} applies; its two exceptions are excluded by hypothesis. Thus an $L$-coloring exists and all $L$-colorings are equivalent.

We may therefore assume that the tight lists are not all identical.

There is a special vertex $v$ of maximum degree. Indeed, if a maximum-degree vertex $u$ were nonspecial, then $L(u)\subseteq L(w)$ for every neighbor $w$ of $u$. Since $|L(u)|=\Delta(G)$ and $|L(w)|=d_G(w)\le\Delta(G)$, equality would hold throughout: $d_G(w)=\Delta(G)$ and $L(w)=L(u)$. If every maximum-degree vertex were nonspecial, connectedness would propagate the same degree and list to every vertex, contrary to nonidentity.

Choose a special color $a\in L(v)$. By \cref{lem:fiber}, the set $\cF(v,a)$ is nonempty and all its members belong to one Kempe class, denoted by $\cK_0$.

Assume for contradiction that a second Kempe class $\cQ$ exists. Fix $\varphi\in\cQ$ and put $p=\varphi(v)$. Since $\cQ\cap\cF(v,a)=\varnothing$, we have $p\ne a$.

Let $B_1(\varphi)=\{u\in N(v):\varphi(u)=a,\ p\notin L(u)\}.$

If $B_1(\varphi)=\varnothing$, continue to the second layer. Otherwise define a degree-assignment $M$ on $G-v$ by
\[
M(u)=
\begin{cases}
L(u)\setminus\{p\},&u\in N(v),\ p\in L(u),\\
L(u)\setminus\{a\},&u\in N(v),\ p\notin L(u),\\
L(u),&u\notin N(v).
\end{cases}
\]

If $G-v$ is $M$-colorable, extend an $M$-coloring by $v=p$. The extension is proper and has empty first obstruction set. The old nonempty first layer witnesses that $p$ is special for $v$, and \cref{lem:fiber} connects the old and new colorings inside $\cQ$.

Suppose $G-v$ is not $M$-colorable. Since $G-v$ is 2-connected, \cref{lem:gallai} makes it either a clique or an odd cycle with one uniform reduced list.

The clique case is impossible. If $G-v=K_m$, noncompleteness gives a nonneighbor of $v$, whereas connectedness gives a neighbor $w$ of $v$. Then $d_G(w)=m>d_G(v)$, contradicting the choice of $v$.

In the odd-cycle case every rim vertex has degree at least three and has only $v$ outside the cycle, so $v$ is universal and $G$ is an odd wheel. Its rim length $m$ is at least five. Let $S$ be the common reduced two-list. A vertex of $B_1$ has original list $S\cup\{a\}$, so $p,a\notin S$. Exactness of the reduction shows that every rim list is either $S\cup\{p\}$ or $S\cup\{a\}$. Since $a$ is special for $v$, some rim vertex $r$ omits $a$ and therefore has list $S\cup\{p\}$.

Because $|L(v)|=m\ge5$, choose $b\in L(v)\setminus(S\cup\{p,a\}).$

No rim vertex lists $b$. The $(p,b)$-component at $v$ is the singleton $\{v\}$, so change $v$ from $p$ to $b$. Now color $r$ with $p$ and alternate the two colors of $S$ along the even-order path obtained from the rim by deleting $r$. This is a proper coloring with $v=b$ and with no rim vertex colored $a$. The color $b$ is special for $v$, so \cref{lem:fiber} connects the first $b$-colored endpoint to this canonical coloring. Finally the $(b,a)$-component at $v$ is the singleton $\{v\}$, and changing it reaches $\cF(v,a)$, contradicting $\varphi\in\cQ$.

Thus $\cQ$ contains a coloring, still denoted $\varphi$, with $\varphi(v)=p$ and $B_1(\varphi)=\varnothing$.

Let $K=K_\varphi(v;p,a)$ and let $B_2$ be the second obstruction set measured inside $K$.

If $B_2=\varnothing$ and swapping $K$ is valid, the swap reaches $\cF(v,a)$. If $B_2=\varnothing$ but the swap is invalid, choose a nearest invalid vertex $t$. Its component distance is at least three. If $y$ precedes $t$ on a shortest path and $\beta=\varphi(y)$, then $\beta\in L(y)\setminus L(t)$, so $\beta$ is special for $y$. Moreover $vy\notin E(G)$. \Cref{lem:pair} connects the present coloring in $\cF(y,\beta)$ to $\cF(v,a)$, again a contradiction.

Hence $B_2\ne\varnothing$. Apply \cref{prop:reduction}. Outcome (i) is covered by the preceding paragraph, and outcome (iii) directly reaches $\cF(v,a)$. Outcome (iv) makes $G$ degree-swappable by \cref{lem:kernels}, so it also reaches this nonempty set. The remaining outcome gives two nonadjacent special vertices with common special color $a$.

\Cref{thm:common} then says that all proper $L$-colorings are equivalent, contradicting the existence of $\cQ$. Therefore every proper $L$-coloring belongs to $\cK_0$.\qedhere 

\end{proof}

\section*{Declaration of competing interest}
The author declares that he has no known competing financial interests or personal relationships that could have appeared to influence the work reported in this paper.

\section*{Acknowledgements}
The research was partially supported by RGC Competitive Earmarked Research Grant 16308821.

\section*{Declaration of generative AI use}
ChatGPT assisted in developing proof arguments, checking intermediate steps, and preparing the manuscript. In particular, this assistance included the maximum-degree reduction in \cref{prop:reduction} and the terminal-block recoloring argument in the proof of \cref{thm:common}. The author takes full responsibility for the accuracy and content of this work.

\end{document}